\documentclass{article}
\usepackage[margin=1in]{geometry}
\usepackage{amsmath,amssymb,amsthm,enumerate,hyperref,xcolor,enumitem, proof, quiver, setspace, comment}
\usepackage[backend=biber,style=authoryear,maxcitenames=2,uniquename=false,uniquelist=false,useprefix=true]{biblatex}
\DeclareBibliographyCategory{cited}
\AtEveryCitekey{\addtocategory{cited}{\thefield{entrykey}}}
\usepackage[parfill]{parskip}
\usepackage{graphicx}
\newtheorem*{theorem}{Theorem}

\newtheorem*{definition}{Definition}

\title{Studying proof data with proof theory: the case of AI-generated mathematics}
\author{Safal R. Aryal}
\begin{document}
\doublespacing
\maketitle
\begin{abstract}
    I present a method for analyzing large collections of formally verified mathematical proofs (e.g. OpenAI's) grounded in 
    some historical and philosophical musings on mathematics in the age of AI. Many of the quantities involved, such as 
    the size of a proof term or the number of redundant steps, are already computed by work on improving 
    automated provers, where they serve as targets. My aim is to turn them into measurements of proofs themselves: 
    compared against a reference library and against each proof's own cut-free form, checked against theorems they 
    must obey, and controlled for how the proofs were compiled. With these instruments, I measure how (much) 
    machine-generated mathematics diverges from human mathematics, and find that AI-written formal proofs realize far less of the logical 
    compression available to them than human mathematicians do (31\% against 73\% in mathlib, and 55\% against 86\% 
    on proofs that call no automation). These figures are broadly in line with reports from those who build provers and 
    proof-improvement tools that AI-written proofs are often longer than they need to be.
\end{abstract}

\section{Introduction}
\begin{quote}
    \textit{
     Rather, as we do not ask what hope of gain makes a little bird warble, since we know that it takes delight in 
     singing because it is for that very singing that the bird was made, so there is no need to ask why the human 
     mind undertakes such toil in seeking out these secrets of the heavens. For the reason why the mind was joined 
     to the senses by our Maker is not only so that Man should maintain himself, which many species of living things 
     can do far more cleverly with the aid of even an irrational mind, but also so that from those things which we 
     perceive with our eyes to exist we should strive towards the causes of their being and becoming, although we 
     should get nothing else useful from them. And just as other animals, and the human body, are sustained by 
     food and drink, so the very spirit of Man, which is something distinct from Man, is nourished, is increased, 
     and in a sense grows up on this diet of knowledge, and is more like the dead than the living if it is touched 
     by no desire for these things. Therefore as by the providence of Nature nourishment is never lacking for 
     living things, so we can say with justice that the reason why there is such great variety in things, and 
     treasures so well concealed in the fabric of the heavens, is so that fresh nourishment should never be 
     lacking for the human mind, and it should never disdain it as stale, nor be inactive, but should have in 
     this universe an inexhaustible workshop in which to busy itself.}
    \hfill --- \textsc{Johannes Kepler}
\end{quote}
\subsection{Up to the stars}

Contemporary AI's seeming resolution of the Navier--Stokes problem \textit{\parencite{openai2026navier}} confirms that machine learning can now generate complex 
mathematical proofs at an unprecedented scale, forcing a critical reflection on the nature of proof itself.
Such reflections (especially those produced in the academy) customarily began with discussions of the foundational crisis that mathematics underwent
in the late 19th and early 20th centuries. But while this was a pivotal chapter in the history of mathematics (and arguably the most 
major moment for proof since Euclid compiled the \textit{Elements} \textit{\parencite{euclid1908}}), I do not believe that this crisis accurately captures 
the epistemological and methodological challenges presented by AI-generated proofs. To fully understand this moment in history, I 
believe that one has to look back further: to Tycho Brahe and Johannes Kepler's collaborative and sequential efforts to collect 
astronomical data during the Scientific Revolution.

For years, Brahe labored at his observatories on the island of Hven to meticulously record decades of naked-eye planetary and stellar 
positions using large, custom-built pre-telescopic instruments (e.g. he used, among many other instruments, a massive brass 
quarter-circle arc mounted to a solid north-south wall with a 2-meter radius, used to measure stellar altitudes as stars 
crossed the meridian) \textit{\parencite{brahe1598,thoren1990}}. When Kepler was appointed as Brahe's successor as the Imperial Mathematician to Holy Roman Emperor 
Rudolf II, he inherited this vast and unprecedentedly accurate observational archive, using the empirical measurements to 
mathematically test and formulate his laws of planetary motion \textit{\parencite{kepler1609,kepler1992}}.

Prior to his death, Tycho Brahe had been puzzled by the motion of Mars: he was not able to find a mathematically parsimonious 
description of how the planet moved that also minimized measurement error. To solve Brahe's puzzle, Kepler had to make five 
key evidentiary innovations. First, he was able to formulate a mathematical method that made calculations
possible that used the position of the true sun (the actual sun you see moving across the sky), rather than the 
mean sun (an imaginary, average sun that moves at a completely steady speed all year--a favorite abstraction of 
the astronomers of antiquity). This allowed for a more accurate view of the relative motions of the Sun, Earth, and Mars, 
eliminating many longitudinal errors that existed in Brahe's interpretations of the data. 

Kepler's second innovation concerned the bisection of eccentricity for the Earth--Sun orbit, which helped reduce the 
discrepancies in Mars' predicted position from over 4 degrees to less than an arcminute. This refinement was made 
possible in part by his third innovation, the area rule: a line joining the Sun and a planet sweeps out equal areas 
in equal periods of time. Together with the use of the true sun, the area rule allowed Kepler to calculate planetary 
positions with remarkable accuracy. Moreover, it constrained the possible orbital geometries: given the temporal and 
positional constraints imposed by Brahe's observations, Kepler found that the ellipse provided the most accurate 
mathematical description of Mars' motion \textit{\parencite{stephenson1987}}. 

None of this, of course, would have been possible without Brahe's dataset. Brahe's data made persistent discrepancies 
available for mathematical investigation at an unprecedented level of precision. Kepler was able to use these discrepancies 
to revise established assumptions about planetary motion, developing a system in which geometrical constructions were 
increasingly informed by physical considerations. This represented a fundamental change in the conception of astronomy, 
from the mathematical description of celestial positions toward the study of physical bodies and their motions. Better known than Brahe or Kepler, Galileo's telescopic observations further advanced this transition: lunar mountains, 
sunspots, the phases of Venus, and Jupiter's satellites challenged important components of Scholastic cosmology and 
expanded the empirical scope of astronomical inquiry \textit{\parencite{galileo1989,galileo2010}}. However, while Kepler had developed remarkably accurate geometrical 
models and Galileo had established new observational facts about the heavens, a unified physical theory of terrestrial and 
celestial motion remained an open problem.

The resolution of this problem, of course, came from Isaac Newton (beginning in \textit{De motu corporum in gyrum} \textit{\parencite{whiteside1974}} and culminating in the 
famous \textit{Philosophiæ Naturalis Principia Mathematica} \textit{\parencite{newton1687,newton1999}}). \textit{De motu} (which our discussion here will focus on) demonstrated how mathematical theory 
could strengthen the explanatory standing of an astronomical system even when competing models accounted for the 
available observations with comparable accuracy. By this point, the quality and resolution of observational evidence
were already considerable, largely due to the work of Brahe and his successors. The remaining problem was to establish a 
more compelling relationship between the observed motions of celestial bodies and the principles of earthly mechanics. 
Newton's innovation in \textit{De motu} was to show that the mathematical structures underlying these two domains could be systematically 
related, thereby providing a stronger basis for preferring the Keplerian system over its competitors.

The significance of Newton's work was that it established precise mathematical relationships between theories that had 
previously developed somewhat independently. By demonstrating how centripetal forces could account for Kepler's laws, 
and deriving further consequences for planetary motion, Newton substantially increased the explanatory power of the laws of  
Kepler's system. These deductions also provided a form of evidential support: a mathematical relationship between 
independently motivated theories could increase the credibility of both, while opening up avenues for further investigation. 
The resulting framework was therefore more robust than its predecessors, insofar as it could account for a wider range of 
phenomena through a common set of principles; a process which culminated in the development of the calculus and Newton's 
laws of motion.

Could the laws of motion have been discovered without these innovations in turning observations into data, and data into 
evidence? Perhaps in some possible world they could have been, but in the world we inhabit, I would wager they could not. 
Astronomy had long been treated primarily as a geometrical discipline, in which the principal task was to formulate 
mathematically coherent descriptions of celestial motion. Brahe's systematic and unusually precise observations made it 
possible to identify discrepancies at a resolution that earlier astronomical models could not adequately accommodate, 
while Kepler's willingness to revise the assumptions through which these observations were interpreted led to a 
substantially more accurate mathematical account of planetary motion. Newton's achievement depended on this prior work, 
since the integration of celestial and terrestrial mechanics required a sufficiently accurate mathematical description 
of orbital motion to establish a meaningful relationship between the two. The transition to Newtonian mechanics thus 
depended on a methodological transformation in which observations were systematically collected, discrepancies were 
treated as grounds for revising existing models, and mathematical relationships were used to establish connections that the 
observations alone could not have revealed. Without these developments, the data necessary for Newton's synthesis might 
have existed in some form, but the evidentiary and mathematical structure through which they became scientifically 
consequential would have been considerably more difficult to establish.

I now turn to argue that mathematics -- the purest, most pristine, deductive science -- is undergoing a similar transformation.
Notwithstanding significant social and economic challenges, we have massive amounts of mathematical data,
which we may turn into mathematical evidence to test mathematical claims.

This paper is a small attempt at doing so. I compiled 96 projects from OpenAI's library of AI-written mathematics,
three smaller AI-written corpora, and two human research projects against the same version of mathlib, Lean's main
library, and measured the proof terms that Lean checks. The structure of proofs has rarely been measured empirically,
let alone at this scale, so part of the exercise is constructing a methodology strong enough to put numbers to it. Many of the numbers
themselves are not new: the people who build and improve automated provers already compute them, as targets to
optimize (see ``What is already measured''). What is new here is the attempt to use them as measurements of proofs. The
measures are checked against theorems they must obey (where such theorems exist), against how much mathlib differs
from itself, and against controls for how the corpora were compiled, and I report where results moved along the way. Three results hold up under these controls.
First, the mix of steps in OpenAI's proofs is 5.8 times as far from mathlib's as random samples of mathlib are from
the rest of mathlib, and still 3.4 to 3.9 times as far once the output of automated tactics is collapsed (though this
comparison does not control for subject matter). Second, OpenAI's proofs state many intermediate claims that do no
work: 40\% of them are larger with their intermediate claims than without them, against 13\% of mathlib's. Third,
they realize less of the compression available to them (the amount by which a proof could be shortened by proving
each repeated piece once and reusing it): 31\% against 73\% in mathlib, and 55\% against 86\% on
proofs that call no automation. Other first impressions did not hold up. OpenAI's proofs are longer mostly because
they call more automation, and once automation is set aside they are no more compressible than mathlib's.
\subsection{The ``foundational crisis'' of the 19th and 20th centuries}
\begin{quote}
     \textit{Already at this time I should like to assert what the final outcome will be: mathematics is a presuppositionless 
     science. To found it I do not need God, as does Kronecker, or the assumption of a special faculty of our understanding 
     attuned to the principle of mathematical induction, as does Poincaré, or the primal intuition of Brouwer, or, finally, 
     as do Russell and Whitehead, axioms of infinity, reducibility, or completeness, which in fact are actual, contentual
     assumptions that cannot be compensated for by consistency proofs.}
     \par\medskip
    \hfill --- \textsc{David Hilbert}, \textit{Die Grundlagen der Mathematik} \textit{\parencite{hilbert1928,hilbert1967}}
\end{quote}
What is mathematics? We can separate this query into two parts: one \textit{metaphysical} and one \textit{epistemic}. 
The former asks what mathematical objects (and mathematics itself) are, while the latter concerns our mathematical 
\textit{knowledge}. While we now conceptualize these as two separate questions, this has not always been the case. 
The contemplation of such questions reached a new height in the so-called ``foundational crisis'' mathematics faced in the 
late 19th and early 20th centuries. We will begin our review of this crisis by comparing views held on this question by two 
bitter rivals: the proto-constructivist Leopold Kronecker, and (putative) father of set theory, Georg Cantor (though do note that it has 
come to light that Cantor may have plagiarized much of his work from Richard Dedekind--since this is meant to be an opinionated 
conceptual summary of the crisis rather than an academic work of history, I will bracket this fact out in the following analysis, but the reader should 
have it in mind).
Consider the famous saying commonly attributed to Kronecker:
\begin{quote}
    \textit{God made the integers, all else is the work of man.} \textit{\parencite[p.~19]{weber1893}}
\end{quote}
Kronecker's mathematical work reflected this position: in number theory and algebra, he sought to ground mathematical 
constructions in the integers and finite arithmetic operations, resisting mathematical objects whose existence depended on 
completed infinite processes. Cantor pursued a sharply 
different program, developing set theory and the transfinite numbers, which treated infinite collections as legitimate 
mathematical objects in their own right, pursuant to the following maxim:
\begin{quote}
    \textit{Mathematics is completely free in its development and is only bound to the self-evident consideration that 
    its concepts are both free of contradictions and that they are in fixed relationships to proven concepts that have 
    already been previously established.} ... \textit{Once a number satisfies all these conditions, it can and must be 
    considered as existing and real in mathematics.} ... \textit{On the other hand, every superfluous constriction of the 
    impulse for mathematical research seems to me to bring with it a much greater danger} ... \textit{for the essence 
    of mathematics lies precisely in its freedom.} \textit{\parencite[\S~8]{cantor1883grundlagen}}
\end{quote}
The underlying disagreement concerns what entitles us to admit a mathematical object into our ontology in the first place. 
For Cantor, the development of mathematics can itself justify an expansion of what exists mathematically: objects introduced 
through new theoretical constructions acquire legitimacy through the demonstrable relationships they bear to existing 
mathematics. This makes mathematical ontology responsive to mathematical practice, insofar as the development of new 
theories can motivate commitments to objects that were not previously recognized. Kronecker's position places greater weight 
on the arithmetic (and ultimately, intuitive -- construct a number system using your fingers!) foundations from which such 
constructions must proceed, thereby restricting the extent to which mathematical practice can determine its own ontology.

It would be glib of me, but not entirely inaccurate, to describe the subsequent episodes in the foundational crisis as footnotes to the 
attitudes set by Kronecker and Cantor. Consider, for example, the work of Gottlob Frege. On the one hand, Frege (at the time 
of \textit{Die Grundlagen der Arithmetik}) appeared to hold Cantor's basic conceptual breakthrough in high regard:
\begin{quote}
    \textit{I find special reason to welcome in Cantor's investigations an extension of the frontiers of science, because they have 
    led to the construction of a purely arithmetical route to higher transfinite numbers (powers).} \textit{\parencite[\S~86]{frege1884,frege1950}}
\end{quote}
However, Frege's own project was to establish a rigorous logical foundation for arithmetic, and this led him to question 
some of the prevailing accounts of what numbers are and how they are derived. In particular, he rejected the idea that 
number could be explained simply by abstracting from the differences between objects and treating them as identical units (as 
Cantor sought to do). His objection was that this procedure seemed to produce a general concept under which objects fall, 
rather than an account of number itself:
\begin{quote}
\textit{And with that we are back once more at our arbitrary way of regarding things, unless we are willing, regardless 
of truth, to ascribe to things an identity going beyond any they actually possess.} ...
\textit{In that event we are not left, as Lipschitz maintains, with ``the concept of the Number of the things considered''; 
what we get is rather a general concept under which the things in question fall.} \textit{\parencite[\S~34]{frege1884,frege1950}}
\end{quote}
Frege's disagreement with David Hilbert brought these questions into particularly sharp focus, culminating in a controversy 
over the foundations of geometry and the role of consistency proofs in mathematical justification \textit{\parencite{frege1980}}. In 
\textit{Grundlagen der Geometrie} \textit{\parencite{hilbert1899}}, Hilbert demonstrated how the consistency of geometrical axioms could be reduced to 
that of arithmetic by interpreting geometrical terms through arithmetic objects and relations. For Frege, axioms 
expressed thoughts about a determinate subject matter, and consistency alone could not establish the existence of the 
objects concerned. Hilbert's approach placed the emphasis on the formal relationships specified by axioms, with 
consistency providing grounds for accepting a mathematical theory. But in another way, his program carried forward Cantor's 
conception of mathematical freedom, developing a systematic method for establishing the legitimacy of mathematical 
theories through their axiomatic structure.

On Kronecker's side, the expression of similar finitist and constructivist inclinations after him perhaps came from 
Émile Borel:
\begin{quote}
    \textit{One can adopt a more strictly arithmetical point of view, as Kronecker did in algebra ... This new point of view can be characterized by the fact that, in any given problem, one introduces only a finite number of integers by means of which all the elements of the problem are explicitly defined.} ...

    \textit{In this construction, the idea of magnitude plays no part at all ... given the procedure itself, we can define only a countable infinity of such numbers. Some minds will see this as a deficiency; others, on the contrary, will think that, in practice, every number, in order to be effectively known, must be capable of being characterized by a finite number of words} ... 
\end{quote}
A related position emerged among several French mathematicians around the turn of the 20th century, including Émile 
Borel, René Baire, and Henri Lebesgue \textit{\parencite{hadamard1905}}, who were later grouped under the label of
\textit{pre-intuitionists} in the terminology of the later Luitzen Jan Egbertus Brouwer \textit{\parencite{brouwer1952}}. Borel's approach placed greater 
restrictions on mathematical existence than Cantor's: consistency alone was insufficient, since mathematical objects 
also had to be effectively definable, in principle through a finite number of words. Given such a restriction, only a 
countable infinity of numbers could be individually characterized, placing substantial limits on the legitimacy of 
arguments involving the uncountable infinite. These mathematicians did not constitute a unified philosophical school, 
and their views differed in important respects. Nevertheless, their work raised a question that Brouwer would pursue much 
further: whether the restrictions required for mathematical knowledge should extend beyond the objects we admit to the 
very logical principles through which we reason about them.

Brouwer's position took the preceding disagreements in a rather different direction by questioning the assumption that 
mathematical objects could be understood independently of the constructions through which they become known. For 
Brouwer, mathematics is fundamentally a creation of the mind, preceding both mathematical language and the logical 
principles used to describe it. Its foundation lies in the intuition of time: the experience of one moment giving way to 
another, while the former is retained in memory. Abstracting from the particular content of this experience produces what 
Brouwer called the intuition of \textit{two-ity} \textit{\parencite{brouwer1952}}, from which the natural numbers arise. On this view, logic describes 
regularities within mathematical activity, and cannot independently establish the existence of mathematical objects.

Brouwer's conception of the real numbers (a mathematical structure that has been the subject of much controversy)
 makes the consequences of this position particularly clear. In his later intuitionistic mathematics, real numbers
 could be represented by infinite sequences of rational approximations, including what he called \textit{choice sequences} \textit{\parencite{brouwer1918}}. 
Some of these sequences are governed by a fixed rule, while others may develop through successive free choices whose 
outcomes are not determined in advance. This provided Brouwer with a conception of the continuum that was not restricted 
to numbers individually characterizable through finite descriptions, as in the approaches associated with Borel. At 
the same time, the infinite sequences involved were understood through their ongoing construction, rather than as 
completed infinite totalities. A real number could therefore have mathematical significance even when its successive 
approximations were not all determined by a single law.

On Brouwer's view, since a mathematical statement is established through a 
construction that proves it, he rejected the unrestricted use of the law of excluded middle, according to which 
every proposition \(P\) satisfies \(P\lor\neg P\). For certain statements concerning real numbers, neither the statement 
nor its negation need be presently constructible. For example, one cannot generally assert that every real number is 
either equal to zero or unequal to zero without a method that establishes one of these alternatives. His approach 
thus connected the epistemology of mathematics directly to its metaphysical commitments: the nature of 
mathematical objects depended on the constructions through which they could be brought into existence.
As expressed in his doctoral thesis \textit{\parencite{brouwer1907,brouwer1975}}:

\begin{quote}
\textit{We start with the notion of the natural numbers, etc. They are so familiar to us, that it is difficult to reduce 
this notion to simpler ones. Yet I shall try to describe their sense in plain words. In the perception of an object we 
conceive the notion of an entity by a process of abstracting from the particular qualities of the object. We also 
recognize the possibility of an indefinite repetition of the conception of entities. In these notions lies the source 
of the concept of natural numbers.}
\end{quote}

Brouwer's position raises a puzzle: if mathematical truth is grounded in mental constructions, what 
becomes of the logical principles traditionally used to establish mathematical truths? This question motivated the work 
of Arend Heyting, a student of Brouwer who sought to give a formal account of intuitionistic mathematics. In 1930, 
Heyting \textit{\parencite{heyting1930a,heyting1930b,heyting1930c}} published a systematic formalization of a logic for intuitionism (something Brouwer himself would likely have been 
skeptical of), including propositional and predicate logic, alongside formal systems for arithmetic and analysis. His 
contribution was to articulate the logical consequences of Brouwer’s philosophical position, showing how mathematical 
reasoning could be governed by precise rules that reflected the constructive character of mathematical knowledge. 
In particular, the law of excluded middle could no longer be assumed indiscriminately, since establishing a 
disjunction requires a construction that establishes one of its alternatives. 

Heyting's work thus introduced a distinction between the formal structure of a proof and the constructive meaning that its 
logical rules were intended to capture. An intuitionistic proof of an implication, for example, requires a construction 
that transforms a proof of its antecedent into a proof of its consequent; a proof of a disjunction requires a proof of at 
least one of its disjuncts. These interpretations connect the validity of logical inferences to the constructions that 
justify their conclusions. In this sense, Heyting provided a framework within which Brouwer's epistemological commitments 
could be investigated through formal logic, even though Brouwer himself regarded mathematical intuition as prior to its 
formal representation. 

Gerhard Gentzen's \textit{\parencite{gentzen1935a,gentzen1935b}} subsequent work brought these concerns into the development of modern proof theory. Through his 
systems of natural deduction and sequent calculus, Gentzen provided formal frameworks for investigating the structure of 
mathematical proofs, including both classical and intuitionistic reasoning. His work on the elimination of cuts, 
in particular, demonstrated that formal derivations possess structural properties that can be studied independently 
of the particular mathematical claims being proved. 
\subsection{Natural deduction, cut elimination, and Hilbert's paradise}
Gentzen's natural deduction systems were developed to formalize mathematical reasoning through assumptions and their 
discharge. To prove an implication $A\to B$, for example, one assumes $A$, derives $B$, and then discharges the assumption:

\[
\frac{
  \begin{array}{c}
  [A]^i\\
  \vdots\\
  B
  \end{array}
}{
  A\to B
}\;(\to I,i)
\qquad
\frac{A\to B\qquad A}{B}\;(\to E).
\]
The introduction rule specifies how an implication can be established, while the elimination rule specifies how it can be 
used. Analogous pairs of rules govern conjunction, disjunction, and the quantifiers. In this way, the structure of a 
mathematical proof can be studied through the rules by which propositions are introduced and used.

However, Gentzen's investigation of proof normalization led him to develop a second formalism, the sequent calculus, 
which provided a particularly powerful framework for analyzing derivations in both classical and intuitionistic logic. 
A sequent has the form
\[
\Gamma\vdash A,
\]
where $\Gamma$ is a collection of assumptions and $A$ is the conclusion. The notation makes explicit the relation between 
assumptions and conclusions that natural deduction tracks through its subproofs. More importantly for our purposes, it 
permits the structure of a proof to be investigated independently of the particular mathematical content of its formulas.
Consider the following rules for conjunction and implication:

\[
\frac{\Gamma\vdash A\qquad\Gamma\vdash B}
     {\Gamma\vdash A\land B}\;(\land R)
\qquad
\frac{\Gamma,A,B\vdash C}
     {\Gamma,A\land B\vdash C}\;(\land L)
\]

\[
\frac{\Gamma,A\vdash B}
     {\Gamma\vdash A\to B}\;(\to R)
\qquad
\frac{\Gamma\vdash A\qquad\Gamma,B\vdash C}
     {\Gamma,A\to B\vdash C}\;(\to L).
\]
Here, the left and right rules describe how the same logical operation functions when a formula occurs among the assumptions 
or as the conclusion. The remaining rules have analogous forms. For intuitionistic logic, the right side of a sequent 
contains at most one formula; the classical sequent calculus permits multiple formulas on the right. This distinction is 
important because the choice of logical system determines which mathematical inferences are admissible. 
In particular, the intuitionistic calculus does not admit the unrestricted law of excluded middle as a theorem.

These positions therefore motivate two questions that proof theory allows us to investigate separately. What can be 
derived from a given mathematical system (a la Cantor, via Hilbert), and what restrictions should be placed on the means 
through which a mathematical conclusion is established (a la Kronecker, via Brouwer)? The classical and intuitionistic 
calculi provide different answers to the second question while permitting a common mathematical investigation of the first. 
The cut elimination theorem is especially significant in this respect, since it establishes a fundamental structural 
property of derivations in both systems.

The rule that the theorem eliminates is known as \textit{cut}:

\[
\frac{\Gamma\vdash A\qquad\Delta,A\vdash C}
     {\Gamma,\Delta\vdash C}\;(\mathrm{Cut}).
\]
This rule permits a proposition established in one derivation to be used as an assumption in another. The proposition 
$A$ disappears from the conclusion, while the assumptions required to establish it are incorporated into the resulting 
derivation. The rule is useful for constructing proofs, since intermediate lemmas can shorten and organize a derivation. 
Gentzen's Hauptsatz establishes that every derivation in the standard sequent calculi for classical and intuitionistic 
logic can be transformed into one that contains no applications of cut.

We prove the intuitionistic version \textit{\parencite{troelstra2000}}. Let $\mathsf{LJ}$ denote a single-succedent sequent calculus for first-order logic, with formulas generated from atomic formulas and $\bot$ using $\land,\lor,\to,\forall,\exists$. Negation is understood as an abbreviation for $A\to\bot$. Contexts are finite sets of formulas, so exchange and contraction are implicit. Weakening is admissible. We use the standard context-sharing presentation of the logical rules; these rules are equivalent to the usual formulations in the presence of the structural properties of the calculus.

The initial sequents and the rule for falsity are

\[
\frac{}{\,\Gamma,A\vdash A\,}\;(\mathrm{Id})
\qquad
\frac{}{\,\Gamma,\bot\vdash C\,}\;(\bot L).
\]

The propositional rules are

\[
\frac{\Gamma,A,B\vdash C}
     {\Gamma,A\land B\vdash C}\;(\land L)
\qquad
\frac{\Gamma\vdash A\qquad\Gamma\vdash B}
     {\Gamma\vdash A\land B}\;(\land R)
\]

\[
\frac{\Gamma,A\vdash C\qquad\Gamma,B\vdash C}
     {\Gamma,A\lor B\vdash C}\;(\lor L)
\]

\[
\frac{\Gamma\vdash A}
     {\Gamma\vdash A\lor B}\;(\lor R_1)
\qquad
\frac{\Gamma\vdash B}
     {\Gamma\vdash A\lor B}\;(\lor R_2)
\]

\[
\frac{\Gamma\vdash A\qquad\Gamma,B\vdash C}
     {\Gamma,A\to B\vdash C}\;(\to L)
\qquad
\frac{\Gamma,A\vdash B}
     {\Gamma\vdash A\to B}\;(\to R).
\]

The quantifier rules are

\[
\frac{\Gamma,A[t/x]\vdash C}
     {\Gamma,\forall x\,A\vdash C}\;(\forall L)
\qquad
\frac{\Gamma\vdash A[y/x]}
     {\Gamma\vdash\forall x\,A}\;(\forall R)
\]

\[
\frac{\Gamma\vdash A[t/x]}
     {\Gamma\vdash\exists x\,A}\;(\exists R)
\qquad
\frac{\Gamma,A[y/x]\vdash C}
     {\Gamma,\exists x\,A\vdash C}\;(\exists L).
\]

Here, $A[t/x]$ denotes capture-avoiding substitution of the term $t$ for the free occurrences of $x$ in $A$. In $\forall R$, the eigenvariable $y$ must not occur free in $\Gamma$; in $\exists L$, it must not occur free in either $\Gamma$ or $C$. Eigenvariables may always be renamed to satisfy these conditions.

\begin{theorem}
If $\Gamma\vdash C$ is derivable in $\mathsf{LJ}$, then it has a derivation containing no applications of $\mathrm{Cut}$.
\end{theorem}

\begin{proof}
Write $h(\pi)$ for the height of a derivation $\pi$, and define the rank of a formula by

\[
\begin{aligned}
r(P)&=r(\bot)=0,\\
r(A\circ B)&=1+\max\{r(A),r(B)\},
  \quad\circ\in\{\land,\lor,\to\},\\
r(Qx\,A)&=1+r(A),
  \quad Q\in\{\forall,\exists\}.
\end{aligned}
\]

We first prove that cut is admissible for cut-free derivations. Suppose that $\pi$ is a cut-free derivation of 
$\Gamma\vdash A$ and $\rho$ is a cut-free derivation of $\Delta,A\vdash C$. We prove that 
$\Gamma,\Delta\vdash C$ has a cut-free derivation by lexicographic induction on
\[
\bigl(r(A),\,h(\pi)+h(\rho)\bigr).
\]
If $\pi$ ends in an initial sequent, then its context already contains $A$, and the desired conclusion follows 
from $\rho$ by admissible weakening. If $\rho$ ends in an initial sequent, either $C$ already occurs in $\Delta$, 
in which case the conclusion is initial, or $C=A$, in which case the conclusion follows from $\pi$ by weakening.

Suppose next that the cut formula is not principal in the last inference of either derivation. The cut can then be 
permuted above that inference. If the inference has multiple premises, we apply the induction hypothesis to each 
premise containing the relevant occurrence and then reapply the inference. Each new cut has the same rank and a 
strictly smaller sum of derivation heights. We may therefore assume that both final inferences are principal on $A$.

Let's now consider propositional cases ($A=B\land D$, $A=B\lor D$, or $A=B\to D$, proofs in the quantificational 
case are similar).
The principal cut is replaced by one or two cuts on the proper subformulas $B$ and $D$. For example, if $A=B\to D$, we 
first cut $\Delta\vdash B$ and $\Gamma,B\vdash D$ to obtain $\Gamma,\Delta\vdash D$, and then cut this 
with $\Delta,D\vdash C$ to obtain $\Gamma,\Delta\vdash C$. For conjunction and disjunction, we similarly cut on the 
relevant subformulas introduced in the premises. In all such instances, the new cut formulas have strictly smaller
rank than $A$, so the induction hypothesis applies to eliminate the new cuts.

Since the rank of $B[t/x]$ is strictly smaller than that of the quantified formula, the induction hypothesis applies to 
eliminate the cut.
Atomic formulas and $\bot$ have no right-introduction rules. They therefore cannot form an additional principal case. 
All possibilities have been exhausted, establishing cut admissibility for cut-free derivations.
Finally, consider an arbitrary derivation, possibly containing cuts. Proceed by induction on its height, replacing 
each immediate subderivation by a cut-free derivation. If its last inference is logical, reapply that inference to the 
cut-free premises. If its last inference is a cut, apply the cut-admissibility result just proved. The result is a cut-free 
derivation of the same end sequent. Hence every derivable sequent in $\mathsf{LJ}$ has a cut-free derivation.
\end{proof}

The significance of the theorem is that cut-free derivations possess a strong structural property. Every formula occurring 
in such a derivation is a subformula of a formula in the end sequent, up to the substitution instances introduced by 
quantifier rules. Intermediate lemmas may considerably shorten a proof, but the cut-elimination theorem shows that they 
are unnecessary for establishing derivability in the calculus. Their use can be eliminated through transformations of 
the derivation itself.

The corresponding result for natural deduction (more elusive) is the normalization theorem. An introduction rule followed immediately 
by its corresponding elimination rule constitutes a detour: for example, one may derive $A\land B$ from proofs of $A$ 
and $B$, and then immediately extract $A$. Normalization replaces this sequence with the original proof of $A$. Cut 
elimination and normalization are closely related, although they are theorems concerning different proof formalisms. 
Gentzen's sequent calculus method was particularly important because it gave a uniform approach to the classical and 
intuitionistic cases.

As fruits of the long march of the foundational crisis, these results give us a mathematical vocabulary for studying the internal structure of proofs. Until recently, many 
questions about that structure were difficult to pursue systematically: proofs were produced to establish particular 
results, and their internal organization was usually examined in the context of individual arguments. AI-generated 
formal mathematics changes the scale of this inquiry. We now have large collections of machine-written derivations 
whose validity can be checked and whose structure can be examined directly.

If the proof assistant settles the question of formal correctness relative to its 
environment, do machine-generated derivations exhibit distinctive structural regularities? How do they compare with those 
produced by human mathematicians? And can the tools of proof theory help us answer these questions empirically? The rest of this
paper suggests that, within limits, they can.

\section{An empirical approach to proof theory}
\subsection{Proofs as programs}
Before we can measure proofs, we need to say what a formal proof is in practice. A \textit{proof assistant} is a computer
program that checks mathematical proofs. The user states a theorem in a precise formal language and writes a proof
of it. The program then checks every step of that proof against a small set of logical rules. If any step does not
follow from the rules, the proof is rejected. The proof assistant we will use to study machine-generated proofs is Lean \textit{\parencite{demoura2021lean4}}.

Lean rests on an idea known as the Curry--Howard correspondence, or \textit{propositions as types} \textit{\parencite{curry1958,howard1980}}. Recall the
natural deduction rules for implication from the previous section. To prove $A\to B$, we assume $A$, derive $B$, and
discharge the assumption. To use $A\to B$, we combine it with a proof of $A$ and obtain $B$. Now consider a function
in a programming language. A function of type $A\to B$ takes an input of type $A$ and returns an output of type $B$.
Writing such a function means assuming a variable $h$ of type $A$ and building something of type $B$ from it. Applying
the function to an input of type $A$ gives an output of type $B$:
\[
\frac{
  \begin{array}{c}
  [h:A]\\
  \vdots\\
  q:B
  \end{array}
}{
  \lambda h.\,q\;:\;A\to B
}\;(\to I)
\qquad
\frac{f:A\to B\qquad p:A}{f\,p\;:\;B}\;(\to E).
\]
Here, $\lambda h.\,q$ is the function that takes $h$ and returns $q$, and $f\,p$ is the function $f$ applied to $p$.
A proposition is a type, and a proof of the proposition is a program of that type. Checking a proof is the same as
checking that a program has the type it claims to have. The same correspondence extends to conjunction, disjunction,
the quantifiers, and much more. Lean uses a rich version of it known as dependent type theory \textit{\parencite{coquand1988}}, in which types may
depend on values (e.g. the type of vectors of length $n$).

In practice, people rarely write these programs by hand. Instead they write \textit{tactics}, which are commands that
build the program step by step. Consider the following small theorem:
\begin{verbatim}
theorem two_steps (a b c : Nat) (h1 : a <= b) (h2 : b <= c) : a <= c + 1 := by
  have h3 : a <= c := Nat.le_trans h1 h2
  exact Nat.le_succ_of_le h3
\end{verbatim}
The theorem says that if $a\leq b$ and $b\leq c$, then $a\leq c+1$. The proof first establishes the intermediate claim
$h_3 : a\leq c$ using transitivity. It then concludes using the fact that anything at most $c$ is at most $c+1$. Lean
reads these tactics and produces a single program, called the \textit{proof term}. This process is called
\textit{elaboration}. A small trusted part of Lean, the kernel, then checks the proof term. The kernel never sees the
tactics.

So, the text of a Lean file is what a person (or a machine) typed. The proof term is the
derivation that the text describes. To get the proof term, the file has to be compiled. It also has to be compiled
against the exact version of the library it was written for, since names and lemmas change between versions.

The library we will compile against when using Lean is mathlib \textit{\parencite{mathlib2020}}. Mathlib is a single large library of formalized 
mathematics, written by hundreds of human contributors over several years. It covers algebra, analysis, topology, number 
theory, and much else. Its proofs are reviewed by other contributors before they are accepted, and it has its own 
conventions for how proofs should be written. At least until recently, almost all serious work in Lean, human or machine, 
was built on top of it--in some sense, it constitutes the ``gold standard'' for Lean proofs (though of course, in some sense, 
it is also an arbitrary choice, as we will see).

\subsection{The data and language}
As discussed, our human baseline is mathlib, at the version used by OpenAI (commit \texttt{d13f23b}) \textit{\parencite{mathlib4}}. The main
AI-written corpus of mathematics we will use for comparison is OpenAI's \texttt{openai/math} library \textit{\parencite{openai_math}}. This repository contains 
524 projects, and we compiled 96 of them, chosen at random (which was about as large a sample I could run on my personal 
MacBook Air). These samples contain about 36,000 theorems. I also compiled three smaller AI-written corpora from other 
systems. The first is a formalization of sphere packing by Math Inc's Gauss \textit{\parencite{spherepacking_lean,hariharan2026spherepacking}}. The second is a set of solutions to the 
2025 Putnam competition by AxiomProver \textit{\parencite{axiom_putnam2025}}. The third is a set of solutions to the 2025 IMO by Aristotle \textit{\parencite{harmonic_imo2025,achim2025aristotle}}. 

Of course, it must be acknowledged (as we began to above) that mathlib is a polished library, and it is not the only kind of
human Lean. We therefore added a second human reference,
made of research projects (the Carleson project \textit{\parencite{carleson_lean,becker2024carleson}} and the formalization of sphere eversion \textit{\parencite{sphere_eversion_lean}}). These were written by
humans for a particular result, much like the AI-written projects (though, compared to the text of mathlib, we suspect 
these libraries are very scarce). To control the quality of the data, the unit of every group comparison is the project. Theorems in the same project share authors and habits, so they are not independent observations. I also only used
projects with at least 20 theorems. When I measure single proofs, I only use proofs of at least 10 steps. Imposing these 
standards, there are about 143,000 proofs in mathlib and 26,000 in OpenAI's library.

The simplest comparison between human (as in, mathlib-written) and AI-generated proofs looks only at their text. A language 
model assigns a probability to each next word (or \textit{token}) of a text, given the words before it. If the model finds 
a text surprising, the text ``costs'' many bits per token. I call this cost, in bits per token, the \textit{energy} of
the text. We trained a simple language model \textit{\parencite{hindle2012naturalness}} (a 5-gram model, which looks 
at the previous four tokens) on mathlib's source code. We then asked how surprising it found OpenAI's source code.

OpenAI's text costs 7.00 bits per token, and mathlib's own (held-out) text costs 5.15. A
sample of OpenAI's text costs more than a sample of mathlib's 94\% of the time. This probability is called the
\textit{AUC} \textit{\parencite{hanley1982auc}}: the chance that a randomly chosen member of one group scores higher than a randomly chosen member of the
other, so that 0.5 means no difference and 1 means complete separation. But
76\% of the gap comes from identifiers. These are the names of lemmas, hypotheses, and definitions. Two proofs can
use the same steps and different names. They can also use different steps and the same names. The text therefore
mostly measures how proofs are named. To measure how they work, we need the proof terms--which, as discussed previously,
are programs.

\begin{figure}[t]
\centering
\includegraphics[width=\linewidth]{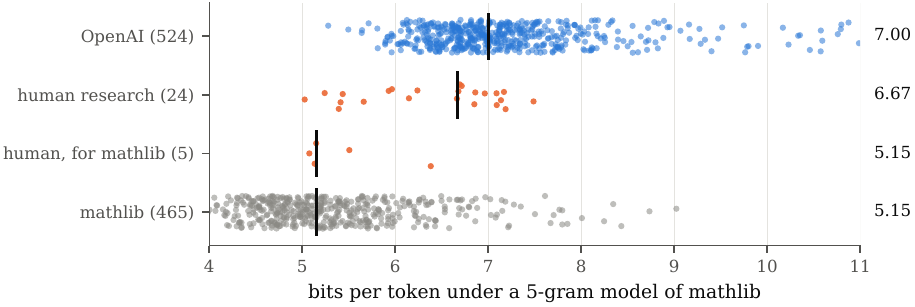}
\caption{Text energy of each project (one dot per project with at least 20 theorems) under a 5-gram model of mathlib's source. Each mathlib project is scored by a model trained on the rest of mathlib. Ticks mark medians; values above 11 bits are drawn at the edge. The row ``human, for mathlib'' holds the directories of the human research projects that were written to be added to mathlib.}
\label{fig:text}
\end{figure}

One might suspect that OpenAI simply writes about different mathematics than mathlib does. This does not explain
the gap. Mathlib and OpenAI's library share 17 areas of mathematics with at least three projects each, and in every one of
them OpenAI's median project costs more. Against projects in its own area, the typical OpenAI project pays 1.53 extra
bits per token. The human research code sits in between (Figure~\ref{fig:text}), at 6.67 bits per token (in the text study, it also includes the FLT and PFR projects \textit{\parencite{flt_lean,pfr_lean}}). Its gap to OpenAI is much
smaller (0.33 bits per token, with an AUC of 0.74). Interestingly, the parts of the human research projects that were
written to be added to mathlib later cost what mathlib costs (5.15, over two projects). So the language model measures how
closely a text follows mathlib's conventions.

\subsection{What is already measured}
Most of the quantities in this paper are not new. They have been computed, often at scale, by the people who build
automated provers and the tools around proof assistants, where they serve as objectives to improve. The first
language-model prover for Lean already found that more than a third of its proofs were shorter, and produced smaller
proof terms, than mathlib's own \textit{\parencite{han2022pact}}. Recent provers
\textit{\parencite{hubert2025alphaproof,ren2025deepseekprover2,wang2025kimina,lin2025goedel,chen2025seedprover}} often write
proofs that are correct but long, and a line of work now shortens them. ImProver rewrites proofs to improve a chosen
metric such as their length \textit{\parencite{ahuja2024improver}}. ProofOptimizer trains a model to simplify proofs,
and more than halves their length on competition benchmarks \textit{\parencite{gu2025proofoptimizer}}. LeanPolish
releases tens of thousands of verified local edits that compress Lean proofs \textit{\parencite{bourigault2026leanpolish}}.
Others refactor generated proofs into reusable lemmas \textit{\parencite{fu2026proofrefactor}} or evaluate their quality
with tools \textit{\parencite{caldwell2026proofjudge}}, and mathlib's own linters flag tactics that change nothing
\textit{\parencite{mathlib4}}. Kallel and El Louadi measure how much kernel-checked proof stands behind the statements
of a release of AI-written results \textit{\parencite{kallel2026verification}}. Older work on interactive theorem
provers measured proof scripts \textit{\parencite{wiedijk2000debruijn,aspinall2016metrics}}, refactored them
\textit{\parencite{whiteside2011refactoring}}, and mined or compressed them by introducing lemmas and definitions
\textit{\parencite{kaliszyk2013lemma,blanchette2015mining,vyskocil2010compression}}. The work on cut introduction
compresses proofs by finding the cuts they lack \textit{\parencite{hetzl2014quantified,ebner2018lemmas}}, and the
realized compression below is a crude relative of it.

What these share is an aim: a shorter proof, a faster check, a better prover. My debt to these prior pieces of work is large because
they shaped a core intuition I had guiding the work in this paper--that AI-generated proofs are somehow large in an ``unnatural''
way (at least for humans), and that it was possible to measure the size of a proof term or the number of its wasted steps as formalized 
in Lean in a rigorous way. What I try to add is a way of using such
quantities as \textit{measurements} of proofs rather than as targets to optimize them against. That means comparing them against how much a reference
library differs from itself, or against the proof's own cut-free form; checking them against theorems they must obey;
and controlling for how the proofs were compiled and for the tactics that wrote them. The hope is that, made reliable
in this way, they can eventually be used to study properties of formal proofs themselves at scale.

\subsection{Preparing the corpora}
Every number in this section depends on how the corpora were read, cleaned, and compiled, so I describe this in some
detail. All of the code for this study (the language model, the Lean programs that extract proof terms, and the
statistics) was written by Claude Code \textit{\parencite{claudecode}}, Anthropic's coding agent, working at my direction. I chose the questions, the
measures, and the controls. The code, the job files, the regression tests, and the result files behind every table
and figure are available in the accompanying repository \textit{\parencite{aryal2026code}}, with instructions for
rerunning them. This way of working has one particular risk: code that is cheap to write is also cheap to
change, and a small change in how a corpus is cleaned can move a result. So the controls had to be maintained
deliberately. Every corpus went through the same pipeline, against the same mathlib, with the same thresholds. When a
part of the pipeline changed, every corpus was rerun through it. Where a
result moved because of such a change, I report the change where the result appears.

The text study works directly on the source files, without running Lean, so it covers all 524 OpenAI projects. The text was prepared as follows:
\begin{enumerate}
    \item
    A lexer (a program that splits source text into tokens) strips comments and string literals, and splits each file into declarations (theorems, lemmas, and
    definitions).
    \item
    Identifiers are split into the pieces mathlib uses to build its names. For example, \texttt{Nat.le\_succ\_of\_le}
    becomes \texttt{Nat}, \texttt{le}, \texttt{succ}, \texttt{of}, \texttt{le}. A new name built from familiar pieces
    is then cheap, and only genuinely new vocabulary is expensive.
    \item
    A piece the model has never seen is coded as an ``unknown'' token followed by its spelling, character by
    character. So the energy of a text is a proper code length: the number of bits needed to transmit it.
    \item
    Numerals of four or more digits are coded digit by digit and booked separately. Otherwise, a few proofs that carry
    long numerical certificates would dominate the cost of their whole project.
    \item
    At most 1,000 declarations are sampled from each non-mathlib project. All of mathlib is used.
    \item
    Mathlib's projects are divided into five folds. Each mathlib project is scored by a model trained on the other four
    folds. Without this, mathlib would be scored by a model that had already seen it, and its energy would be
    artificially low.
    \item
    The model is an interpolated Kneser--Ney model \textit{\parencite{kneser1995,chen1999smoothing}}, a standard smoothed $n$-gram model. Its order (five) was chosen
    among orders three to six by its cross-entropy on held-out mathlib text.
\end{enumerate}
Each token is also classed as a name piece, a numeral, or an operator. This is how the gap was split above. Of the
pooled gap of 1.83 bits per token, 76\% is carried by name pieces, and syntax adds only 0.28 bits.

The compiled study needs much more care, since a proof term exists only once Lean has elaborated the file. The
pipeline is as follows:
\begin{enumerate}
    \item
    All corpora are compiled with Lean 4.34.1 against mathlib at commit \texttt{d13f23b}, which I built on my machine.
    OpenAI's full library has 26 million lines, which is too large to compile with 8 GB of memory, hence the sample.
    \item
    A Lean program (the \textit{driver}) elaborates each file and walks every proof term. For each theorem it writes
    out the gestures, the cut records, and the proof-term measures described below.
    \item
    A \textit{module} is a single Lean file, and its \textit{environment} is the set of definitions and notation it
    can see. The OpenAI projects were compiled in one shared environment, one module after another. This is faster than
    Lean's normal build, which gives each module an environment made only of its own imports. In a shared environment,
    notation declared in one module can leak into the next. The human projects were compiled module by module, each
    with its own imports, as Lean's build tool would. To check whether the
    compilation mode matters, I compiled nine OpenAI projects both ways. The gestures were identical for 97.2\% of 4,089
    theorems, and the divergence moved from 280.4 to 280.5 bits.
    \item
    Several engineering failures had to be handled without biasing the sample. Lean never frees memory for repeated
    imports, so the module-by-module driver restarts every eight modules. A module that stalls is skipped after a time
    limit and recorded; one of the 97 drawn OpenAI projects stalled for more than 45 minutes and is excluded. An early
    version of the driver also used a Lean interface that silently drops error messages, which could have hidden
    failing theorems. It was replaced.
    \item
    Theorems containing \texttt{sorry} (Lean's placeholder for a missing proof) are dropped everywhere. Theorems that do
    not compile under this mathlib are dropped, at the rates given below.
    \item
    Lean generates auxiliary theorems automatically, such as equation lemmas for definitions and the
    \texttt{\_proof\_n} lemmas it splits off long proofs. These are folded back into the declaration that produced them,
    so that a proof is not counted twice because the compiler cut it in two.
    \item
    As discussed below, \texttt{omega} and \texttt{lia} are merged into one gesture.
\end{enumerate}
I also checked a sample of extracted proofs by hand against their source.

At the end of this pipeline, each theorem is represented by its proof term, the program that Lean's kernel actually
checks. The question is then what to measure on these terms, given that the corpora prove different theorems.

\subsection{Gestures}
A proof term is a tree. Each node is one step of the proof, such as the use of a lemma, the application of a function,
or the introduction of an assumption. The children of a node are the subproofs it depends on. Most of the detail at
each node is specific to the theorem being proved. But, we want to compare proofs of different theorems, so we must find a way to strip this
detail away.

\begin{definition}
We assume a proof term is a tree.
Then, define a \textit{gesture} as a node of a proof term with three changes made to it. One, for a gesture, all names of 
hypotheses are erased. Then, every type is cut down to its head connective (e.g. an equation becomes $=$, an implication becomes 
$\to$). Finally, every hypothesis is labelled by how often it is used: $0$ (never), $1$ (once), or $\omega$ (more than once).
\end{definition}

For example, the transitivity of equality takes a proof of $a=b$ and a proof of $b=c$ and returns a proof of $a=c$.
Every use of it becomes the same gesture,
\[
\mathtt{Eq.trans}(=,=)\Rightarrow\;=,
\]
whatever $a$, $b$, and $c$ are.

Erasing names and keeping only the shape of each step is Hilbert's attitude toward geometry, applied to proofs. For Hilbert, it did not matter whether the axioms spoke of points
and lines or of tables and chairs \textit{\parencite{blumenthal1935}}; what mattered were the relations the axioms laid down. A gesture likewise keeps the
relations between the steps of a proof and throws away what the steps are about. Of course, Frege's objection applies \textit{\parencite{frege1980}}
here too. Two proofs with the same gestures may be about entirely different things. We accept this loss because the corpora prove
different theorems, so the shape of their steps is the only thing that can be compared across them.

We walk every proof term in full and record the gesture at each node. Each project then has a \textit{gesture mix},
the proportion of its steps taken up by each gesture. To compare corpora, I kept the 300 most common gestures and gave
every project equal weight (so that a few large projects cannot dominate). We then compare the spread of OpenAI's
projects with the spread of mathlib's using the Kullback--Leibler divergence \textit{\parencite{kullback1951}}, which measures in bits how different two
distributions are. If each corpus's gesture mixes are approximated by a normal (Gaussian) distribution, this divergence splits exactly into a sum over a set of directions.
Each direction is a particular combination of gestures in which OpenAI's projects vary differently from mathlib's. I
call these directions \textit{eigenproofs}. The bits carried by the eigenproofs add up to the total divergence.

Any two sets of projects differ somewhat, and a sample of 96 projects differs from its population by chance alone. So
the divergence has to be compared with how much mathlib differs from itself. To measure this, we draw random sets of 96 mathlib
projects, and measure their divergence from the rest of mathlib. This is the \textit{null}.

OpenAI's gesture mix is 280 bits from mathlib's. Random sets of mathlib projects of the same size are 49 bits away (and
at most 58 bits, in ten draws). OpenAI's library is therefore 5.8 times as far from mathlib as the null. Human research
code is 1.8 times as far from mathlib as its own null. About a third of OpenAI's divergence, and most of human research
code's excess over its null, comes from the certificates that tactics write into proof terms (see ``Is it only tooling?''). Figure~\ref{fig:spectrum} shows how the 280 bits are spread
over the eigenproofs. Each of the first twelve exceeds the largest value the null produced at its rank. A theorem's
\textit{score} on an eigenproof is the projection of its gesture mix onto that direction. The table gives the share of
theorems whose score lies more than two standard deviations (s.d.) above mathlib's mean. The first four eigenproofs
are:
\begin{center}
\begin{tabular}{llcccc}
 & & & \multicolumn{3}{c}{theorems beyond $+2$ s.d.} \\
 & eigenproof & bits & mathlib & OpenAI & human \\
\hline
1 & integer arithmetic inside chains of equations & 23.6 & 0.7\% & 6.3\% & 0.5\% \\
2 & rewriting under binders & 17.5 & 0.2\% & 4.2\% & 4.4\% \\
3 & closing by evaluating numerals & 15.5 & 0.6\% & 8.4\% & 11.2\% \\
4 & normalizing, then closing & 13.5 & 0.9\% & 13.5\% & 12.8\% \\
\end{tabular}
\end{center}
Only the first eigenproof belongs to OpenAI alone. It is the pattern of proving an equation by a call to an arithmetic
decision procedure (Lean's \texttt{omega} \textit{\parencite{pugh1991}}) in the middle of a chain of equations. The second rewrites inside a
\textit{binder}, a $\lambda$ that introduces a variable. The others are habits that OpenAI shares with human research
code, to a similar degree. When the data is split in half at random, the first direction
is recovered on each half with a correlation of 0.95.

\begin{figure}[t]
\centering
\includegraphics[width=\linewidth]{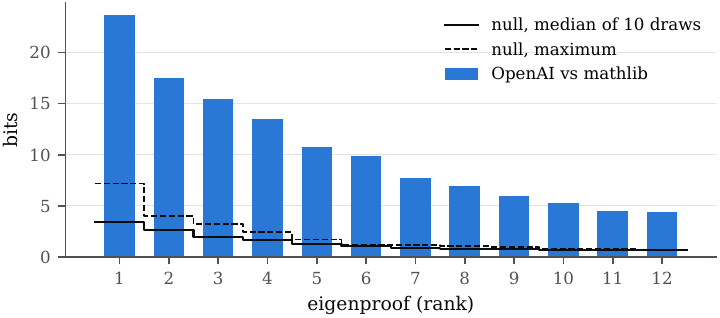}
\caption{Bits carried by OpenAI's first twelve eigenproofs against mathlib, with the null: the same decomposition for random sets of 96 mathlib projects against the rest of mathlib (median and maximum over ten draws).}
\label{fig:spectrum}
\end{figure}

One choice affects these numbers. Mathlib recently switched from \texttt{omega} to a newer tactic, \texttt{lia},
which does the same job. OpenAI's library still uses \texttt{omega}. If the two are counted as different gestures, the
divergence rises from 280 to 316 bits. I merged them, since the two tactics decide the same class of problems. The 36
bits between the two figures come entirely from a change in mathlib's tooling.

The gestures also let us count some simpler things. Each hypothesis is used never, once, or more than once, and each
intermediate claim is a cut (we will return to cuts below). The medians over projects are:
\begin{center}
\begin{tabular}{lccccc}
 & mathlib & OpenAI & human & AUC vs mathlib & AUC vs human \\
\hline
steps per theorem & 18.5 & 27.0 & 29.0 & 0.75 & 0.40 \\
unused hypotheses & 8\% & 12\% & 14\% & 0.64 & 0.41 \\
intermediate claims per 100 steps & 0.36 & 1.20 & 0.64 & 0.86 & 0.79 \\
unused claims & 0.0\% & 0.9\% & 1.0\% & 0.74 & 0.50 \\
\end{tabular}
\end{center}
Against mathlib, OpenAI's proofs are longer, state more claims, and leave more hypotheses unused. The claims and the
unused hypotheses survive a control for automation, but the length mostly does not (see ``Is it only tooling?''). Against human research
code, only the extra claims stand out (AUC 0.79). Human research code is as long (AUC 0.40) and leaves as many
hypotheses unused (AUC 0.41).

\subsection{What mathlib confounds}

So far, every comparison has been a comparison with mathlib. As discussed, mathlib is in some sense the gold standard
for Lean proofs. But, nonetheless, it cannot be relied on its own to measure the \textit{intrinsic} properties of mathematical
proofs--in fact, at least five of its properties could each produce a difference between OpenAI and mathlib that has nothing to
do with machines.

\begin{enumerate}
     \item 
     Mathlib has a \textit{house style}. Its proofs are reviewed, shortened, and rewritten to follow its conventions
     before they are accepted. A proof that is not written the way mathlib writes proofs costs more under any model of
     mathlib, whether or not it is a worse proof. Energy therefore measures conformity to mathlib. 
     The \texttt{omega} and \texttt{lia} example above is a small instance of this.
     \item 
     Mathlib has a different \textit{purpose}: it is a library, and its lemmas are stated in general form so that others
     can use them. OpenAI's projects (like the Putnam and IMO solutions) are aimed at particular results. A research project
     leaves hypotheses unused and states claims it does not need, because it is unfinished.
     The human research code controls for this, at least partly. As we saw, unused hypotheses and dead claims are as common
     in Carleson and sphere eversion as in OpenAI's library. They may therefore be features of research formalization in
     general, rather than of machines. Of course, the human reference is small: 11 compiled projects from two efforts, against 1,193 mathlib
     projects. Its null is correspondingly large, and every interval around it is wide.
     \item
     Mathlib is tied to a \textit{version}. Every corpus has to be compiled against the same mathlib to be compared.
     The other AI corpora were written for older versions of Lean (4.20 to 4.28), and some of their theorems no longer
     compile. Sphere packing loses 20\%, Putnam 11\%, and IMO 67\%. Even Carleson, written by humans, loses 222 of its 4,898
     theorems. These failures are biased: longer proofs are more likely to break, so the surviving samples lean towards
     shorter proofs.
     \item
     A proof term records the \textit{tools} as much as the author. When a proof calls \texttt{omega}, the tactic
     writes a \textit{certificate} into the proof term (a proof that its answer is correct), and that certificate has
     its own gestures. Our first eigenproof is, in part, a measure of how often a tactic is called. I separate the two below (``Is it only
     tooling?''), by collapsing each certificate into a single step.
     \item
     Finally, we must also consider the question of \textit{subject matter}. Mathlib covers some areas much more heavily than OpenAI does. The text study controls
     for this by comparing within areas, as above. The gesture study does not, and it should be read with this in mind.
\end{enumerate}
All five of these confounds share one structure. Each of them comes from measuring a proof against \textit{other} proofs. 
A gesture mix is only unusual relative to some population, and the population I chose was mathlib. If I had chosen 
differently, the numbers would change. Brahe measured positions against the fixed stars, which are not truly fixed, and 
his data was still usable. But a measure that needs no reference population would remove most of these confounds at once (the tools, as we
will see, remain). I now argue that cuts (and their eliminations, a la Gentzen) give us one.
\subsection{Cuts in Lean}

Recall the cut rule:
\[
\frac{\Gamma\vdash A\qquad\Delta,A\vdash C}
     {\Gamma,\Delta\vdash C}\;(\mathrm{Cut}).
\]
A cut proves an intermediate claim $A$ and then uses it as an assumption. This is exactly what the \texttt{have}
tactic does in the example above. The proof first establishes $h_3 : a\leq c$, and then uses $h_3$ to prove the
theorem. Under propositions as types, the proof $p$ of the claim and the proof $q$ that uses it combine into the term
\[
(\lambda h.\,q)\;p,
\]
a function that takes $h$ and returns $q$, applied to $p$. Lean has several ways of writing this (\texttt{have}, a
\texttt{let} that names a proof, or a function applied directly to a proof). We count all of them as cuts.

Eliminating a cut means removing the intermediate claim. In terms of programs, we replace every occurrence of $h$ in
$q$ by $p$ itself, written $q[p/h]$. For the example above, the cut-free proof term is
\[
\mathtt{Nat.le\_succ\_of\_le}\;(\mathtt{Nat.le\_trans}\;h_1\;h_2).
\]
How much this costs depends on how often $h$ is used, which is exactly what the labels $0$, $1$, and $\omega$ in our
gestures record. If $h$ is never used, the cut is \textit{dead}, and eliminating it simply deletes $p$. If $h$ is used
once, the cut is \textit{linear}, and $p$ moves to the place where it is used. Then, we say that if $h$ is used many
times, the cut is \textit{shared}, and $p$ is copied once for each use. It seems that only a shared cut can make a proof smaller.

\subsection{Don't eliminate cut!}
Gentzen showed that every cut can be eliminated. But, Statman and Orevkov \textit{\parencite{statman1979herbrand,statman1979lambda,orevkov1982}} later showed that the cost of doing so can be
enormous. In fact, in the worst case, the cut-free proof is larger than any fixed tower of exponentials in the size of the
original.

George Boolos drew a philosophical lesson from this in a short paper titled \textit{Don't eliminate cut} \textit{\parencite{boolos1984}}. He gave an
inference in first-order logic that is valid, and that has a short derivation in a standard natural deduction system.
The derivation fits on a few pages, and it works by proving a handful of lemmas and then using them. Boolos then showed
that any derivation of the same inference in a standard cut-free system is unimaginably long. It could not be written
down in the physical universe, let alone checked. His point was that cut elimination makes lemmas dispensable only in principle. The cut-free derivation is a finite
object, so even Kronecker would grant that it exists. But in Boolos's example, only the derivation with lemmas can be
written down and checked.

This gives us a way of looking at cuts that does not depend on mathlib. Every proof has a cut-free form, and we can ask of each cut whether it earns its place, by comparing the proof with its own
cut-free form. A gesture mix is unusual only relative to a population. A cut is judged
against the proof that contains it. In this sense, I would argue, cuts are more \textit{intrinsic} to a proof
than gestures are.

Of course, Boolos was concerned with the extreme case, in which a cut saves a tower of exponentials. We will see below
that real libraries are nowhere near this case. Most cuts in mathlib are linear (74\%), so they organize a proof
without shortening it. The measures below ask, proof by proof, whether the author's cuts shorten it, and how much
shorter it could have been.

\subsection{Measuring cuts}

The \textit{size} of a proof term is the number of nodes in its tree. For each proof $\pi$, we compute three sizes.
The first is the size of the proof as written, $|\pi|$. The second is the size of its cut-free form,
$|\pi^{\mathrm{cf}}|$, with every cut eliminated. The third is the size of its optimal layout, $|\pi^{\mathrm{opt}}|$.
This is the same proof with the best possible choice among the author's cuts. It keeps every cut that saves space and
removes every cut that does not. (A cut used $u$ times with a proof of size $|p|$ is kept when
$(u-1)\cdot|p| > u+1$, that is, when the copies it saves outweigh the cost of naming it.)

The cut-free form can be very large. Some proofs in mathlib consist of hundreds of nested cuts, and their cut-free forms
are $2^{300}$ to $2^{450}$ times larger than the proofs as written. So we never build it. Instead we use
\textit{hash-consing} \textit{\parencite{filliatre2006hashcons}}, which stores each distinct subterm only once. A subterm that appears a thousand times is stored
once and counted a thousand times. This gives us the exact size, as a very large integer, without writing the proof
out.

\begin{definition}
The \textit{cut compression} of a proof $\pi$ is
\[
\kappa(\pi)=\log_2\frac{|\pi^{\mathrm{cf}}|}{|\pi|}.
\]
Its \textit{compressibility} is
\[
\chi(\pi)=\log_2\frac{|\pi^{\mathrm{cf}}|}{|\pi^{\mathrm{opt}}|},
\]
and its \textit{layout redundancy} is
\[
\rho(\pi)=\log_2\frac{|\pi|}{|\pi^{\mathrm{opt}}|}.
\]
\end{definition}

All three are measured in bits. A proof that is half the size of its cut-free form has $\kappa=1$. The optimal layout is
never larger than the other two, so $\chi$ and $\rho$ are never negative. The three are related by
\[
\kappa=\chi-\rho.
\]
The compressibility $\chi$ is how much a proof \textit{could} save by its cuts. The redundancy $\rho$ is how far the
author's choice of cuts is from the best one. If $\rho>\chi$, then $\kappa<0$, and the proof is larger with its cuts
than without them. The \texttt{have} in our small example is a case of this. The claim $h_3$ is used once, so it saves
nothing, and the extra step only adds to the size.

The quantity $\kappa$ also has a direct reading. When we eliminate every cut, each step of the written proof is copied some
number of times. One can show that $\kappa$ is exactly $\log_2$ of the average number of copies of a written step. So
one bit of cut compression means that, on average, each step would be written twice without the cuts.

\subsection{Computing the cut measures}
The definitions above leave several implementation choices open, and each one affects the numbers. I made them as
follows:
\begin{enumerate}
    \item
    \textit{What counts as a cut.} A cut is a \texttt{have}, a \texttt{let} that binds a proof, or a function applied
    directly to a proof. Whether a term is a proof is decided from the declared type of its statement. A function applied
    to data (e.g. to a number) is substituted without being counted as a cut, because a proof cut can be hidden
    inside it.
    \item
    \textit{What counts as a use.} A use is an occurrence of the cut's variable in a proof. A proof that is mentioned
    only inside a type or a piece of data is counted as a single copy, since eliminating the cut never duplicates
    it. An early version counted every such mention as a copy, which inflated the use counts. This was
    found and fixed before any result was reported.
    \item
    \textit{Five normal forms.} A \textit{normal form} here is a version of the proof with some class of cuts
    eliminated. For each proof, the driver computes the exact size of five versions: as written
    ($N_0$); with dead cuts removed ($N_1$); with linear cuts also removed ($N_2$); with every cut removed ($N_3$, the
    cut-free form); and in the optimal layout ($N_4$). The optimal layout is decided from the innermost cut outwards,
    keeping a cut used $u$ times with a proof of size $|p|$ exactly when $(u-1)\cdot|p|>u+1$. The five sizes split
    $\kappa$ into three parts,
    \[
    \kappa=\underbrace{\log_2\frac{N_3}{N_2}}_{\sigma}
    -\underbrace{\log_2\frac{N_1}{N_2}}_{\ell}
    -\underbrace{\log_2\frac{N_0}{N_1}}_{w},
    \]
    where $\sigma$ is the gain from shared cuts, $\ell$ the cost of naming proofs that are used once, and $w$ the cost
    of proofs that are never used.
    \item
    \textit{Exact sizes, averaged per proof.} Sizes are exact integers of arbitrary size. A corpus figure is the mean,
    over proofs, of each proof's bits. A ratio of summed sizes would be decided by a single mathlib proof whose
    cut-free form is $2^{450}$ times its written size. The first version of this analysis used such ratios, and was
    corrected before publication.
    \item
    \textit{Rates.} Rates such as ``cuts per 100 steps'' divide by the size of the proof with shared subterms counted
    once, so that proofs whose trees are enormous through repetition do not dilute every rate.
\end{enumerate}
The realized compression of the next subsections needs a second, independent run of the driver over the proof terms,
which I call the \textit{hashing pass}. It eliminates every written cut by substitution, and stores the result only through hashes. Every subterm receives a
\textit{canonical hash}, a fingerprint computed from its structure, so that identical subterms receive identical
hashes. A cut variable hashes as the proof it stands for, and an ordinary bound variable hashes as its binder. So the
same subproof is a single node wherever it occurs, whether it was written out twice, named once and used twice, or
substituted from a cut. Two occurrences of a subproof are identified only if their free variables (variables bound outside the subproof)
are bound by the same binders. This ensures that every node shared in the count could actually be named by a \texttt{have} at that
point of the proof. The explicit layout $G$ names a node when its number of references times its size exceeds its size
plus its references plus one, and is then refined greedily up to six times. Each proof is given a budget of 200,000
distinct nodes, and proofs over budget are dropped (10 of 254,147 in mathlib). Logical rank unfolds definitions up to
24 times.

Before the hashing pass was run on any corpus, it was checked against a set of small theorems whose values were
computed by hand. These remain in the repository as a regression test. The pass went through four versions. The
first read ranks syntactically and failed the rank ladder; the second fixed the two errors behind this (described
under ``Checking the instruments''); the third also unfolds definitions when deciding whether a hypothesis is a
proof; the fourth fixed the last exceptions to the ladder (also described there). Every corpus was run with the fourth version.

\subsection{Cuts across corpora}
Table values are means over proofs of at least 10 steps. Comparisons between projects use the projects with at least
20 such proofs (94 of OpenAI's 96). Figure~\ref{fig:chirho} shows the same two quantities per
project.
\begin{center}
\small
\begin{tabular}{lcccccc}
 & mathlib & human & OpenAI & sphere packing & Putnam & IMO \\
\hline
proofs measured & 143,448 & 3,195 & 25,928 & 3,553 & 353 & 29 \\
proofs with a cut & 17\% & 33\% & 47\% & 54\% & 73\% & 69\% \\
cuts dead / linear / shared & 1/74/25\% & 1/77/22\% & 2/82/16\% & 3/88/9\% & 4/86/10\% & 2/86/12\% \\
$\chi$ (bits per proof) & 0.029 & 0.063 & 0.035 & 0.022 & 0.056 & 0.056 \\
$\rho$ (bits per proof) & 0.015 & 0.028 & 0.071 & 0.188 & 0.174 & 0.159 \\
proofs larger with their cuts & 13\% & 26\% & 40\% & 50\% & 67\% & 55\% \\
\end{tabular}
\end{center}
AI-written proofs contain more cuts, and more of their cuts are linear. Comparing projects, OpenAI's redundancy $\rho$ is
higher than mathlib's with an AUC of 0.95 ($p=2\cdot 10^{-44}$; all $p$-values are from two-sided Mann--Whitney
tests \textit{\parencite{mann1947}}), and higher than human research code's with an AUC of
0.85 ($p=4\cdot 10^{-5}$). Almost all of this redundancy comes from claims used exactly once. OpenAI's compressibility
$\chi$ is only slightly higher than mathlib's (AUC 0.64, $p=1\cdot 10^{-5}$), and human research code is the most compressible of the three (OpenAI against
human, AUC 0.22, $p=9\cdot 10^{-4}$). OpenAI's extra cuts therefore come from claims its proofs do not need. Even the slight excess in $\chi$ comes from
automation (see ``Is it only tooling?''). On proofs that call no automation, OpenAI is no more compressible
than mathlib (AUC 0.55, $p=0.14$), and its redundancy is as high as ever (AUC 0.94). Sphere packing and Putnam are more redundant still (AUC 0.95 and 1.00 against mathlib), and
neither is clearly more compressible than mathlib (AUC 0.58 and 0.70, $p=0.42$ and $0.03$). The IMO corpus is too small to test.

\begin{figure}[t]
\centering
\includegraphics[width=0.8\linewidth]{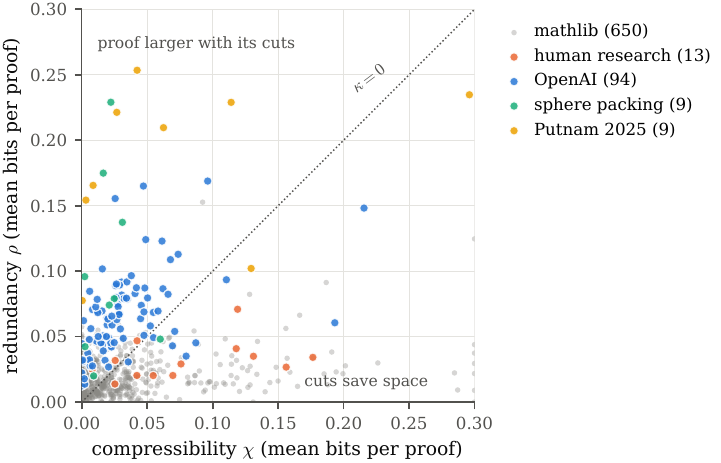}
\caption{Mean compressibility $\chi$ and redundancy $\rho$ per project (projects with at least 20 proofs). Above the diagonal, $\kappa<0$ on average, and the project's proofs are larger with their cuts than without them. Values above 0.3 are drawn at the edge.}
\label{fig:chirho}
\end{figure}

All of these values are small. Sharing could halve a proof (that is, $\chi\geq 1$) in at most 2.3\% of proofs in any
corpus. None of these corpora comes near the regime of Boolos's example.

I also measured cuts one level up, treating a project's own lemmas as cuts and their citations as uses. This result
did not survive more data. With the first 67 OpenAI projects, OpenAI seemed to compress more at this level than mathlib.
With all 94, the difference was no longer significant ($p=0.11$). Inlining every local
lemma multiplies the median project by about $2^{1.1}$ in mathlib and $2^{1.5}$ in OpenAI.

\subsection{Is the divergence only layout?}

OpenAI's redundant claims add steps, and every step has a gesture. So it is fair to ask whether the 280 bits of
divergence come only from the redundant claims. To check, I reran the entire gesture
analysis (vocabulary, eigenproofs, null, and tests) on three subsets. The first leaves out definitions, which have no
cuts. The second keeps only proofs already in their optimal layout ($\rho=0$). The third keeps only proofs with no cuts
at all. These keep 30,157, 18,220, and 17,881 of OpenAI's 36,336 theorems.

\begin{figure}[t]
\centering
\includegraphics[width=0.9\linewidth]{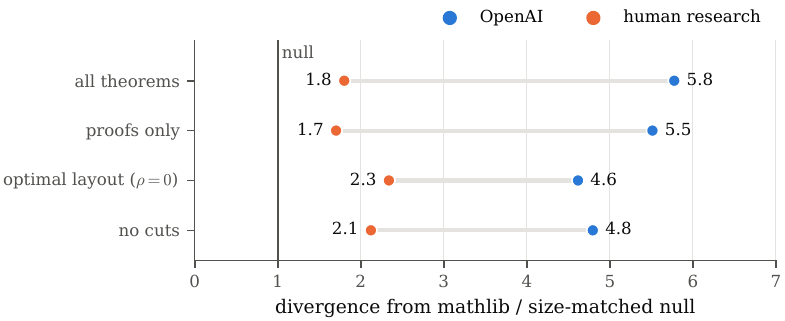}
\caption{Divergence from mathlib as a multiple of the size-matched null, on all theorems and on the three layout-controlled subsets. Each subset reruns the whole pipeline (vocabulary, eigenproofs, null). OpenAI's divergence is 280, 299, 308, and 316 bits against nulls of 49, 54, 67, and 66.}
\label{fig:layout}
\end{figure}

The divergence survives (Figure~\ref{fig:layout}). On proofs in their optimal layout it is still 4.6 times the null. The first eigenproof, however, is
mostly layout: 76\% of OpenAI's excess on it comes from proofs with redundant layout. The typical
case is a line like \texttt{have h : a = b := by omega}, a claim proved by arithmetic and used once. On proofs in their
optimal layout, this direction drops to fifth place (13 bits), and the top direction becomes rewriting under binders
(46 bits). Proofs with redundant layout carry 43\% of OpenAI's excess on this direction, close to their 40\% share of
OpenAI's proofs. Even on these proofs, OpenAI scores beyond $+2$ s.d. on
eigenproofs 2 to 8 between 13 and 31 times as often as mathlib does.

The same check changed one earlier result. A Kneser--Ney model trained on mathlib's proof terms predicts each gesture
from its parent in the tree, and so assigns each proof a cost in bits per step. Over all theorems, this model found
OpenAI's proofs no more surprising than mathlib's (3.27 bits per step for both, AUC 0.53). On proofs in their optimal layout,
OpenAI costs more (3.37 against 3.13, AUC 0.63). The redundant claims are cheap and predictable steps, and they had
been pulling OpenAI's average down. I had reported the tie before running this check. On the compiled sample, the text gap between median projects moves from 1.75 to 1.78 bits per token.

\subsection{Realized compression}

The redundancy $\rho$ measures waste. But $\chi$ only measures the compression available among the author's own cuts.
A proof may contain a repeated subproof that the author never named, and $\chi$ cannot see it. So we would also like to
know how much compression a proof contains at all, whoever laid it out, and how much of it the author obtains.

For this, I eliminate every cut and then look for repeated subproofs in the result. Let $S$ be the size of the
cut-free proof. Then define three quantities, all in bits. One, $\chi_{\mathrm{author}}=\log_2(S/A)$, where $A$ is the
size of the author's own cuts in their optimal layout. This is what the author's cuts achieve. Two,
$\chi_{\mathrm{lb}}=\log_2(S/G)$, where $G$ is the size of an explicit layout that names any repeated subproof that
pays for its own name, whether or not the author named it. Finally, $\chi_{\mathrm{DAG}}=\log_2(S/D)$, where $D$ counts
the distinct subproofs, that is, the nodes of the proof's \textit{directed acyclic graph} (DAG), in which every
repeated subproof is stored once. This is sharing with no cost for naming. If $\chi^*$ is the best achievable compression, then
\[
\chi_{\mathrm{author}}\leq\chi^*
\qquad\text{and}\qquad
\chi_{\mathrm{lb}}\leq\chi^*\leq\chi_{\mathrm{DAG}}.
\]
The \textit{realized share} of a corpus is
\[
\frac{\sum_\pi\chi_{\mathrm{author}}(\pi)}{\sum_\pi\chi_{\mathrm{lb}}(\pi)}.
\]
Since $\chi_{\mathrm{lb}}$ is at most $\chi^*$, this share is a conservative estimate. Importantly, $S$, $D$, and
$\chi^*$ depend only on the cut-free proof. Any two layouts of the same proof (e.g. \texttt{have h := p; (h, h)} and
\texttt{(p, p)}) give the same values.
\begin{center}
\small
\begin{tabular}{lcccccc}
 & mathlib & human & OpenAI & sphere packing & Putnam & IMO \\
\hline
proofs measured & 143,439 & 3,195 & 25,927 & 3,553 & 353 & 29 \\
$\chi_{\mathrm{author}}$ & 0.029 & 0.063 & 0.035 & 0.022 & 0.056 & 0.056 \\
$\chi_{\mathrm{lb}}$ & 0.039 & 0.105 & 0.115 & 0.162 & 0.186 & 0.173 \\
$\chi_{\mathrm{DAG}}$ & 0.100 & 0.223 & 0.290 & 0.424 & 0.418 & 0.437 \\
realized share & 73\% & 61\% & 31\% & 14\% & 30\% & 33\% \\
missed at least 1 bit & 0.1\% & 0.3\% & 1.1\% & 3.2\% & 3.4\% & 3.4\% \\
\end{tabular}
\end{center}
The hashing pass was a separate run, and it drops proofs over its budget, so its proof counts differ slightly from
those in the first table.
Mathlib realizes 73\% of the compression its proofs contain, and human research code 61\%. OpenAI realizes 31\%, Putnam
30\%, and sphere packing 14\%. Sphere packing is the AI corpus whose gestures are closest to human code
(1.4 times its null), and it realizes the least. Part of these gaps comes from automation (see ``Is it only
tooling?'').

AI-written proofs also have more compression available ($\chi_{\mathrm{DAG}}$ is two to four times mathlib's). This is
mostly because they are longer and call more automation, whose certificates repeat themselves. On proofs that call no
automation, OpenAI's $\chi_{\mathrm{DAG}}$ is 1.5 times mathlib's. From 100 nodes up, proofs of the same size have about the same $\chi_{\mathrm{DAG}}$ in
every corpus, to within a factor of about 1.6. One might then object that longer proofs are simply harder to lay out
well. To check this, I grouped proofs by the size of their cut-free form (Figure~\ref{fig:size}).

\begin{figure}[t]
\centering
\includegraphics[width=0.9\linewidth]{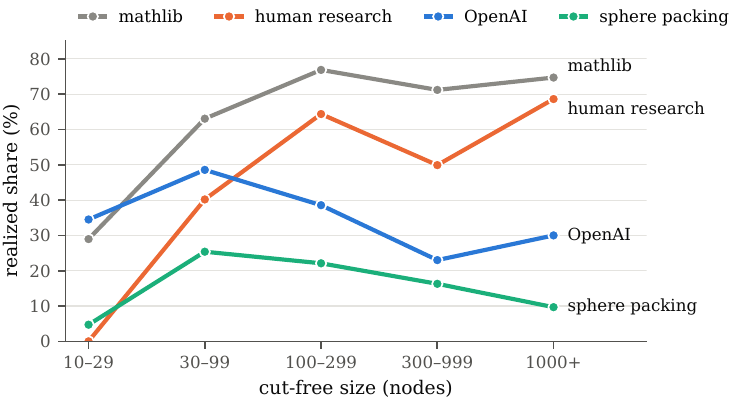}
\caption{Realized share within bins of cut-free proof size. The bins hold 89,144, 42,600, 8,393, 2,264, and 906 mathlib proofs; 10,301, 9,219, 3,148, 2,121, and 1,049 OpenAI proofs; 1,402 to 115 human proofs; and 1,263 to 285 sphere packing proofs. Putnam and IMO are left out, since most of their bins hold fewer than 30 proofs.}
\label{fig:size}
\end{figure}

From 30 nodes up, mathlib realizes more in every bin. Across projects, OpenAI's realized share is lower than mathlib's
with an AUC of 0.25 ($p=6\cdot 10^{-15}$), and of 0.36 on proofs that call no automation. OpenAI therefore loses in both directions, even on proofs that call no automation. It writes cuts that save
nothing (the redundancy $\rho$), and it leaves out cuts that would save something.

Mathlib also leaves compression unrealized. Across all its proofs, the author's cuts reduce the cut-free size by 36\%, the explicit
layout by 48\%, and full sharing by 55\%. Its largest misses are computational. The proof that the Cartan matrix of
$E_8$ has determinant 1 has a cut-free size of 561,653 nodes, and its author's layout saves nothing. An explicit layout
would bring it to 21,052 nodes, and full sharing to 8,156. The same holds for $E_6$ and $E_7$ and for several formulas
on Weierstrass curves. In each case a tactic emits the same subproof many times.

Two subproofs are counted as shared only if they are identical. A
mathematician's lemma is usually a \textit{general} statement, with a parameter, used at several different instances.
Such lemmas are invisible to this measure, so realized compression captures repetition only. Abstraction of this kind
is the target of the literature on cut \textit{introduction} (e.g. the work of Hetzl, Leitsch, Reis, and Weller \textit{\parencite{hetzl2012towards,hetzl2014quantified}}). Measuring it at this scale
remains open.

\subsection{Is it only tooling?}

We asked above whether the divergence is only layout, and the same question can be asked of the tools. As noted among the confounds, a tactic such as \texttt{ring} or
\texttt{omega} writes a certificate into the proof term, and the certificate has its own gestures. Certificates make up
32\% of OpenAI's gestures and 9\% of mathlib's. The rewriting that \texttt{rw} and \texttt{simp} leave behind (chains of
\texttt{Eq.mpr}, \texttt{congrArg}, and \texttt{Eq.trans}) is the same kind of artifact. To check, I collapsed each
certificate into a single step, which records that a tactic was called, and kept everything else, including the facts
the author passes to the tactic. I then reran the entire gesture analysis (vocabulary, eigenproofs, null) three times:
with automation collapsed, with automation and rewriting collapsed, and with both removed, so that only the author's own
steps remain.

OpenAI's divergence falls from 280 to 182 bits, against a null of 49 (3.7 times). With rewriting collapsed as well, it
is 191 bits against 56 (3.4 times), and on the author's own steps, 220 bits against 56 (3.9 times). Each of these is
far above the largest value the null produced (60 to 68 bits). So about a third of the divergence comes from the tools,
and the rest does not. Human research code behaves differently: its divergence falls from 1.8 to 1.1 times its null,
within the range of the null itself, so most of what separated it from mathlib was tooling. For OpenAI,
with automation collapsed, the first eigenproof is how often a tactic is called, together with lemmas about
inequalities. On the author's own steps, it is chains of inequality lemmas (from $\leq$ to $\leq$, and from $<$ to
$\leq$). This may be subject matter rather than style, since the gesture study does not compare within areas.

For the other measures, the simplest control is to keep only the proofs that call no automation at all. OpenAI calls
automation in 21\% of its proofs, against 4\% of mathlib's, but it rewrites about as much. The AUCs of OpenAI against
mathlib, on all proofs and on proofs that call no automation, are (proofs of at least 10 steps, in projects with at
least 20 such proofs, so the first column differs slightly from the earlier tables):
\begin{center}
\begin{tabular}{lcc}
 & all proofs & no automation \\
\hline
steps per theorem & 0.87 & 0.73 \\
unused hypotheses & 0.75 & 0.77 \\
intermediate claims per 100 steps & 0.94 & 0.95 \\
redundancy $\rho$ & 0.95 & 0.94 \\
compressibility $\chi$ & 0.63 & 0.55 \\
$\kappa$ & 0.12 & 0.07 \\
realized share & 0.25 & 0.36 \\
\end{tabular}
\end{center}
The redundant claims, the unused hypotheses, and the low $\kappa$ do not depend on the tools. OpenAI's proofs are longer
mostly because they call more automation, and the slight excess in $\chi$ comes from automation alone ($p=0.14$ without
it), since certificates repeat themselves. The realized share moves the most, because certificates hold most of the compression
available in the proofs that call automation, and no author can name a subproof inside one. On proofs that call no
automation, mathlib realizes 86\% of the available compression, human research code 82\%, OpenAI 55\%, and sphere
packing 33\%. OpenAI still realizes less than mathlib (AUC 0.36, $p=1\cdot 10^{-5}$) and less than human research code
(AUC 0.28, $p=0.01$), but the gap is smaller than over all proofs, and human research code no longer differs from
mathlib (AUC 0.49).

\subsection{Checking the instruments}

Some of what our instruments should report is fixed by theorems. A measurement that violates a theorem is a broken
measurement, so I treat the following results as calibration.

First, by the reading of $\kappa$ above, a step can be copied at most as often as the product of the uses of the cuts
around it. This gives two bounds:
\[
\kappa\leq\sum_{\text{cuts}}\log_2(\text{uses}),
\qquad
\chi_{\mathrm{DAG}}\leq\log_2(\text{most copies of any subproof}).
\]
Both hold for every proof in every corpus. The first bound is loose (in mathlib, $\kappa$ is typically 0.18 of it),
since it ignores how large the copied proofs are. A refined estimate that accounts for size,
$\sum\log_2(1+(u-1)\cdot|p|/N)$ over cuts (with $u$ the cut's uses, $|p|$ the size of its proof, and $N$ the written
size), tracks $\kappa$ with a Spearman rank correlation \textit{\parencite{spearman1904}} of 0.95 in mathlib.

Second, the enormous blowups found by Statman and Orevkov have a known source. A \textit{redex} is a function applied to an argument, which is how a cut appears in a proof term. When a cut is
eliminated, a proof can be substituted into a position where it is applied to another proof. This produces a new
redex, called a \textit{created redex}. Define the \textit{rank} of a cut as the depth of implications nested to the left in
its formula (with $\neg A$ read as $A\to\bot$). A created redex can only arise from a cut of rank at least one. Cuts of
rank zero cost at most a single exponential when eliminated. I call this the \textit{rank ladder}.

When I first checked the rank ladder, however, it failed for more than half of the mathlib
proofs with created redexes. A theorem cannot fail, so the instrument had to be wrong. By printing a counterexample, I
found two errors. The pass read ranks from the written formula, so that $s\subseteq t$ or \texttt{Monotone f}, which are
implications once their definitions are unfolded, counted as rank zero. It also missed cuts hidden behind
substitutions of data. Once both were fixed, the ladder held for 99.6\% of 1,100 mathlib proofs with created redexes.
I printed the five remaining exceptions (four in mathlib, one in the IMO corpus), and each had a cut of rank one that the
pass read as rank zero. Two cuts are inclusions of submodules, $p\leq q$, and one is an inclusion of finite sets; their
orders are built with \texttt{let}, which the unfolding did not pass through. One cut's hypothesis is a proposition
only through a defined family (the arrows of a sieve), which the pass unfolded before reading its type. The last
mentions a set defined by a local \texttt{let}, which the pass did not substitute. With these fixed, the ladder holds for
every proof with created redexes in every corpus (1,100 in mathlib, and 35, 492, 50, 14, and 5 in the human, OpenAI,
sphere packing, Putnam, and IMO corpora).

Third, real libraries are far from the worst case. Sharing could halve a proof ($\chi_{\mathrm{DAG}}\geq 1$) in only 2.1\% of mathlib proofs. Created
redexes appear in 0.8\% of mathlib proofs and 1.9\% of OpenAI's. With definitions unfolded, 21\% of mathlib's cuts have
rank at least one (against 8\% when ranks are read syntactically). In mathlib, proofs whose shared cuts reach rank one
or more have \textit{less} compression available than those with only rank-zero shared cuts ($\chi_{\mathrm{DAG}}$ of
0.50 to 0.59, against 0.74; Figure~\ref{fig:rank}). What rises with rank is created redexes, from 5\% to 30\%. A shared
cut of rank one is a local lemma of the form $A\to B$ that is applied at several places. Positive rank is the
precondition for Boolos-type speed-ups. In mathlib, the proofs that contain positive-rank shared cuts are less compressible than proofs whose shared cuts all
have rank zero, which suggests that such cuts serve abstraction rather than compression.

\begin{figure}[t]
\centering
\includegraphics[width=\linewidth]{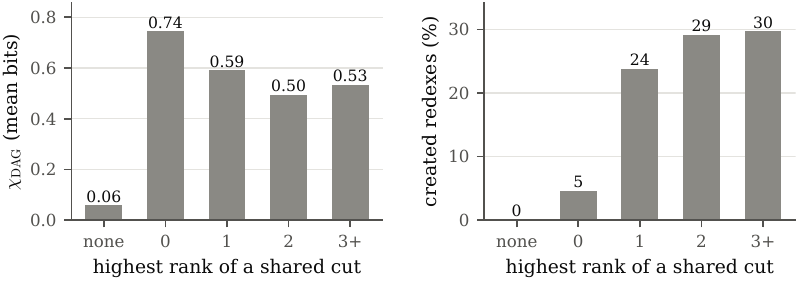}
\caption{Mathlib proofs grouped by the highest logical rank among their shared cuts (definitions unfolded): 134,172 with no shared cut, then 6,752, 1,992, 385, and 138. Left: available compression. Right: share of proofs in which eliminating cuts creates new proof redexes.}
\label{fig:rank}
\end{figure}
Finally, the pipeline was checked against itself. The driver's first run and the hashing pass compute cut-free
sizes independently, the first by normalizing each proof and the second by hashing its term. They agree for 98.5\% of
mathlib proofs, and the exceptions are the created redexes, which the first run reduces and the hashing pass leaves
in place. The OpenAI projects compiled both ways (see ``Preparing the corpora'') also give the same proof terms: of
2,879 proofs from seven of them, 99.8\% have identical sizes, and the 5 that differ do so outside their cuts.

\section{From mathematical data to mathematical evidence?}
Tycho Brahe's observations of the heavens were valuable because their precision made discrepancies impossible to ignore. 
Their provenance mattered too, in the long run especially: Kepler inherited measurements that could be revisited, 
recalculated, and brought to bear on competing models. It would be rather grandiose to cast our collection of Lean files as 
a new observatory on Hven, and I do not mean to suggest that we are on the verge of a Newtonian synthesis. Still, a compiled Lean file gives us something unusually definite: a proof term checked by the kernel 
against a proposition in a specified formal environment. This establishes that the term proves the proposition Lean 
has elaborated, relative to the declarations and axioms in play. This may be less than it seems: the formal 
proposition may fail to capture the intended mathematical claim; the accompanying informal argument may say 
something different; the proof may depend on a particular library version or environment that others cannot 
readily reproduce. If these artifacts are to become useful mathematical data, we must keep track of where they 
came from, what they depend on, and which attempts to compile them failed.

OpenAI's public mathematics repository (whose openness was 
crucial to the success of this project) contains hundreds of manuscripts grouped into families, while its README \textit{\parencite{openai_math}}
acknowledges that the results occupy different stages of verification and that roughly 42\% of the top-line results have 
Lean formalizations. Moreover, on 6 October 2026, Bastounis, Circelli, and Hansen \textit{\parencite{bastounis2026}} 
published a preprint arguing that parts of OpenAI's announced Navier--Stokes proof do not correspond faithfully to their 
Lean formalizations. One of their examples concerns an inverse estimate for which the informal paper requires four additional 
derivatives while the cited formal result uses five; the authors also identify discrepancies in a pressure-flux 
estimate and its proof. The authors distinguish 
their claims about formalization mismatches from a judgment on the correctness of the informal proof. Kallel and El Louadi make a related 
argument from another release of AI-written results, in which the statements a human must audit are a small fraction 
of the kernel-checked proof \textit{\parencite{kallel2026verification}}. These criticisms 
illustrate a rather immediate difficulty for empirical work: a formal verification result can be sound relative to its 
environment while failing to certify the claim for which it is presented as evidence.

The problem becomes even more acute when such artifacts are collected at scale. A corpus that treats compiled proofs, 
uncompiled drafts, incomplete formalizations, and proofs checked under incompatible environments as equivalent 
observations will obscure precisely the distinctions it is meant to investigate. Opaque or redundant derivations 
also impose costs on subsequent researchers, who must recover their dependencies, determine what has actually been 
established, and distinguish mathematical structure from accidental features of the generation or formalization process. 
The opposite policy, discarding every proof that fails to compile against a chosen library version, introduces its own
bias. In the corpora examined here, moving to a common Lean and mathlib environment eliminates 20\% of the 
sphere-packing theorems, 11\% of the Putnam solutions, and 67\% of the IMO solutions; 222 of the 4,898 theorems in the 
human-written Carleson project also fail to survive the migration. Longer proofs are more likely to break, so the 
surviving corpus systematically favors shorter ones. A dataset assembled without recording these losses could 
therefore mistake compatibility with a particular software environment for a structural property of mathematical reasoning.

This is where empirical analysis becomes part of the solution. In the present study, the rank-ladder test initially 
failed on more than half of the mathlib proofs to which it applied. Since the bound follows mathematically, these 
failures exposed defects in the measurement procedure. Inspecting the counterexamples revealed errors in the 
treatment of defined propositions and substitutions; correcting the instrument restored the expected bound in every 
tested case. The same principle extends to the corpus itself: unusual proof structures, unexplained changes across 
versions, excessive intermediate claims, and abnormal patterns of sharing can identify artifacts that merit closer 
inspection. Of course, such analyses cannot replace checking the correspondence between a formal statement and its intended 
mathematical meaning. They can, however, make that audit more systematic and help identify where the underlying data or the 
instruments used to study it have gone wrong. In my ideal world, the aim would be to build a mathematical archive in which 
failures and corrections remain visible, so that accumulated proofs become a resource for discovering regularities rather 
than a mass of outputs whose evidential status is impossible to recover.

Most abstractly, underlying this enterprise is the siren's call to Cantor's paradise \textit{\parencite{hilbert1926}}. Once a proof has been formalized, 
it becomes possible to collect it alongside thousands of others, strip away much of its mathematical content, and compare 
structures that would otherwise be difficult to place on common footing. This is what the notion of a gesture attempts 
to do. We erase names, retain the head connective of each type, and record how often each hypothesis is used. The result 
is a description of a proof that can be compared with proofs of entirely different theorems. Naturally, we lose something 
in the process; Frege would have his objection ready. Even so, the abstraction separates the corpora: OpenAI's gesture distribution lies 
280 bits from mathlib's, against 49 bits for a typical size-matched sample of mathlib compared with itself. The 
divergence persists when we restrict attention to proofs in their optimal layout, and when each call to automation is
collapsed into a single step. We have constructed a common space in 
which a difference between collections of proofs becomes measurable, even though the theorems they establish may have 
little in common.

But, the Kroneckerian impulse enters at the point where we insist on seeing the construction itself. A theorem 
statement in a manuscript is one thing; a proof term that Lean's kernel has checked is another. This insistence 
imposes costs: to compare the corpora, I compiled them against a common version of Lean and mathlib, and some proofs 
did not survive the journey. Moreover, longer proofs were more likely to fail, so the surviving data were 
biased towards shorter proofs. Here the demand for an effectively exhibited construction collides with the ambition to 
study mathematics in its full variety. We cannot simply declare every surviving proof representative of its original 
corpus, nor can we conclude that every discarded proof was mathematically defective. Some failures reflect changes in 
libraries and toolchains. The exactness of our criterion does not make the selection it imposes innocuous--and we can always 
hunt for more fine-grained ways to classify our proofs.

There is a similar tension that appears within the proof terms themselves. We can ask how much a proof contracts when 
its intermediate claims are eliminated, and how much repeated structure its author has managed to exploit. These questions 
admit exact answers for the formal objects we study, even when normalization would produce a term far too large to write 
down explicitly. Yet the answers resist any simple celebration of machine-generated mathematics--OpenAI's proofs contain 
more intermediate claims than either mathlib or the human research projects, but those claims often do little work: 
40\% of OpenAI's proofs are larger with their cuts than without them, against 13\% in mathlib and 26\% in the human 
research corpus. OpenAI also realizes only 31\% of the compression available through repeated subproofs, against 73\% 
in mathlib (and 55\% against 86\% on proofs that call no automation). Of course, these are merely properties of the derivations we measured, not verdicts on the mathematical worth of their 
conclusions. But then the whole point of the exercise is precisely that such distinctions can be made, tested, and revised. 
The present archive gives us something to investigate, and we can investigate it with eyes towards its limitations.

To me, the most interesting prospect of this line of inquiry is that this study will never quite finish collecting its 
object. A proof changes its appearance as we move between source text, elaborated term, cut-free form, and compressed 
representation; a library changes beneath it, and an apparently innocuous change in our instruments can send us back 
to reconsider what we thought we had measured. This makes the empirical study of proofs a rather peculiar science, 
whose data are themselves mathematical constructions and whose instruments must answer to mathematical theorems. 
There is something here that vindicates Imre Lakatos's insights in \textit{Proofs and Refutations} \textit{\parencite{lakatos1976}}: the proof is a moving 
object, altered by criticism, counterexample, and the introduction of distinctions that were invisible in the original 
formulation. Perhaps, as these corpora grow, we will discover regularities in mathematical reasoning that are difficult 
to see from within any single proof, or find that some of our present categories collapse under closer inspection. Either 
way, the study will have to proceed through conjectures about proofs, counterexamples to our conjectures, and revisions 
to the concepts with which we describe them. The archive will change as mathematics changes, and our account of 
mathematical practice will have to change with it.

\printbibliography[category=cited,heading=bibintoc,title={References}]
\nocite{*}
\printbibliography[notcategory=cited,heading=bibintoc,title={Further related work}]

\end{document}